\documentclass[english, 11pt]{amsart}
\usepackage[margin=3.5cm]{geometry}

\usepackage{amssymb}
\usepackage{amsfonts}
\usepackage{amscd}
\usepackage[T1]{fontenc}
\usepackage{color,mdwlist, enumerate}

\title[A $P$-minimal structure without Skolem functions]{An example of a $P$-minimal structure without definable Skolem functions}

\author{Pablo Cubides Kovacsics}

\address{Pablo Cubides Kovacsics \\ Laboratoire de math\'ematiques Nicolas Oresme\\ Universit\'e de Caen\\CNRS UMR 6139 
Universit\'e de Caen BP 5186\\
14032 Caen cedex, France. }

\email{pablo.cubides@unicaen.fr}

\author{Kien Huu Nguyen}

\address{Kien Huu Nguyen \\ Laboratoire Paul Painlev\'e\\ Universit\'e de Lille 1\\ CNRS U.M.R. 8524\\ 59655 Villeneuve d'Ascq Cedex, France.
\\ and \\
Hanoi National University of Education \\
136 XuanThuy str., CAU GIAY 
\\ Hanoi, Vietnam
}
\email{hkiensp@gmail.com}

\newtheorem{theorem}{Theorem}[section]
\newtheorem*{theorem*}{Theorem}

\newtheorem{proposition}[theorem]{Proposition}
\newtheorem*{prop*}{Proposition}

\newtheorem{lemma}[theorem]{Lemma}
\newtheorem{claim}[theorem]{Claim}

\newtheorem{def-theorem}[theorem]{Theorem-Definition}
\newtheorem*{statement*}{Statement}

\theoremstyle{definition}

\newtheorem*{thm*}{Theorem}
\newtheorem{question}{Question}
\newtheorem{definition}[theorem]{Definition}

\def\rv{\operatorname{rv}}
\def\RV{\operatorname{RV}}
\def\VF{\operatorname{VF}}
\newcommand{\Lm}{\mathcal{L}}
\newcommand{\Lring}{\Lm_{\text{ring}}}

\def\NN{{\mathbb N}}
\def\QQ{{\mathbb Q}}

\def\ZZ{{\mathbb Z}}

\def\cF{{\mathcal F}}

\def\cL{{\mathcal L}}
\def\cM{{\mathcal M}}

\def\cO{{\mathcal O}}

\begin{document}
\begin{abstract}
We show there are intermediate $P$-minimal structures between the semi-algebraic and sub-analytic languages which do not have definable Skolem functions. As a consequence, by a result of Mourgues, this shows there are $P$-minimal structures which do not admit classical cell decomposition.     
\end{abstract}


\maketitle

%\tableofcontents

In this article we provide an example of a $P$-minimal field without definable Skolem functions. In Section 1, we give a general introduction to $P$-minimality in which we explain the relevance of such an example within its study. The construction of the example together with some final comments will be presented in Section 2. 

\section{Preliminaries}

Hereafter $K$ will denote a $p$-adically closed field, that is, a field elementarily equivalent to a finite extension of $\QQ_p$ in the language of rings $\Lring:=(+,-,\cdot,0,1)$. We use the notation $\Gamma_K$ for the value group of $K$, $v:K\to\Gamma_K\cup\{\infty\}$ for the $p$-adic valuation and $\cO_K$ for the valuation ring of $K$. For a language $\cL$, by $\cL$-definable sets we mean definable by an $\cL$-formula allowing parameters. We sometimes drop the prefix $\cL$ and say \emph{definable} when the ambient language $\cL$ is clear from the context. 

%Definable sets (resp. $\cL$-definable sets if we want to stress the ambient language) are definable sets with parameters (resp. $\cL$-definable sets with parameters). 

Introduced by Haskell and Macpherson in \cite{haskellETAL:10}, $P$-minimality is a model-theoretic tameness notion for $p$-adically closed fields. It was inspired by \emph{o-minimality}, a similar tameness notion which was initially developed for real-closed fields. As a consequence of their work (see Theorem 2.2 in \cite{haskellETAL:10}), the following definition can be taken as a variant of their original formulation.  
\begin{definition}
Let $\cL$ be a language extending $\Lring$ and $K$ be a $p$-adically closed field. The structure $(K,\cL)$ is $P$-minimal if for every structure $(K',\cL)$ elementarily equivalent to $(K,\cL)$, every $\cL$-definable subset $X\subseteq K'$ is $\Lring$-definable. 
\end{definition}

%The reader familiar with o-minimality may think of $P$-minimality as a $p$-adic analog but should be careful not to take this analogy too far. Despite the many similarities, they can be quite different in some respects, as we will see later.

By results of Macintyre in \cite{macintyre:76}, which were later extended by Prestel and Roquette in \cite{pre-ro-84}, $p$-adically closed fields in $\Lring$ are $P$-minimal. Another source of examples of $P$-minimal fields comes from adding analytic structure to a given $p$-adically closed field. Let us mention some of these examples. 

Recall that Macintyre's language $\cL_{mac}$ is $\Lring$ extended with unary predicates $P_n$ for each integer $n>0$, which are interpreted in $K$ by the sets of $n^{\text{th}}$-powers $P_n := \{y^n \in K : y \in K^\times\}$. For $K$ a finite extension of $\QQ_p$, the subanalytic language $\cL_{an}$ on $K$ is the language $\cL_{Mac}$ enriched with the field inverse $^{-1}$ extended to $K$ by setting $0^{-1}=0$ and, for each convergent power series $f: \mathcal{O}_K^n\to K$, a function symbol for the restricted analytic function
\begin{equation*}
x\in K^n\mapsto
\begin{cases}
f(x) & \text{ if } x\in \mathcal{O}_K^n\\
0 & \text{otherwise.}
\end{cases}
\label{eq:subanalytic}
\end{equation*}

By a result of Haskell, Macpherson and van den Dries in \cite{HaskellMacVDD:99}, $(K,\cL_{an})$ is $P$-minimal. Variants of $\cL_{an}$ for non-standard $p$-adically closed fields (i.e., $K$ not being a finite extension of $\QQ_p$) were shown to be $P$-minimal by Cluckers and Lipshitz in \cite{CLip}. 

\

Something which all previously given examples have in common is that they all satisfy cell decomposition and cell preparation. Before discussing how these notions are related to the existence of definable Skolem functions, let us remind the reader what cell decomposition and cell preparation mean for us in this article. We will work relative to a given collection of functions $\cF$. The first step is to define cells over $\cF$:

\begin{definition}[Cells]\label{def:cells} Let $\cF$ be a family of functions on $K$. 
\item An $\cF$-cell $A \subseteq K$ is a (nonempty) set of the form
\begin{equation*}
\{t \in K : v(\alpha) \ \square_1 \ v(t - c) \ \square_2 \ v(\beta), t - c \in \lambda P_n\}
\end{equation*}
with $\lambda, c \in K$, $\alpha,\beta\in K^\times$, and $\square_i$ either $<$ or $\emptyset$ (i.e., `no condition').  
\item An $\cF$-cell $A \subseteq K^{m+1}$, $m \geq 1$, is a set of the form
\begin{equation*}
\{(x, t) \in K^m\times K : x\in D, v(\alpha(x)) \ \square_1 \ v(t - c(x)) \ \square_2 \ v(\beta(x)),
t - c(x) \in \lambda P_n\},
\end{equation*}
with $D= \pi(A)$ an $\cF$-cell (where $\pi:K^{m+1}\to K^m$ denotes the projection onto the first $m$ coordinates) and 
$\alpha,\beta : K^m \to K^\times$ and $c : K^m \to K$ functions in $\cF$. We call $c$ the
center and $\lambda P_n$ the coset of the $\cF$-cell $A$.
\end{definition} 

%Notice that the dependence on $\cF$ is only relevant for cells $A\subseteq K^{m+1}$ when $m>0$ and for $m=0$ we only include `$\cF$' in the definition to make work the induction. 
We can now give a definition of cell decomposition and cell preparation relative to a family of functions $\cF$.  

\begin{definition}\label{def:celldecomp} Let $\cF$ be a collection of functions. 
\begin{enumerate}
\item $(K,\cL)$ has \emph{cell decomposition over $\cF$} if every definable set can be partitioned into finitely many $\cF$-cells. 
\item $(K,\cL)$ has \emph{cell preparation over $\cF$} if given definable functions $f_j:X\subseteq K^{m+1}\to K$ for $j=1,\ldots,r$, there exists a finite partition of $X$ into $\cF$-cells $A$, such that if $A$ has center $c : K^m \to K$ and coset $\lambda P_n$ with $\lambda\neq 0$, for each $(x, t) \in A$
\end{enumerate}
\begin{equation*}
v(f_j (x, t)) = v(\delta_j (x))+\frac{a_j(v(t - c(x))-v(\lambda))}{n} \hspace{1cm} \text{for each $j = 1, \ldots , r$},
\end{equation*}
with $a_j$ an integer, and $\delta_j : K^m \to K$ a function in $\cF$. If $\lambda = 0$ we just have that $v(f_j (x, t)) = v(f_j (x, c(x))) = v(\delta_j (x))$. When $m=0$, a function $K^0\to K$ is assumed to be a single element of $K$. 
\end{definition} 

Clearly, if $(K,\cL)$ has cell preparation over $\cF$ it also has cell decomposition over $\cF$. What is classically referred to as semi-algebraic (resp. analytic) cell decomposition, would, in our notation, correspond to cell preparation over the class of continuous $\Lring$-definable functions (resp. over the class of analytic functions which are definable in $\cL_{an}$). Denef proved, in his foundational article \cite{denef-86}, that $p$-adically closed fields $(K,\Lring)$ have semi-algebraic cell decomposition. Cluckers in \cite{clu-2003}, showed that $(K,\cL_{an})$ has analytic cell decomposition. Similar results were proven in \cite{CLip} by Cluckers and Lipshitz for variants of $p$-adically closed fields with analytic structure. 

It is therefore natural to ask whether similar theorems could be generalized to $P$-minimal structures, a question which was already raised in \cite{haskellETAL:10}. A first (partial) answer was given by Mourgues in \cite{mou-09}, where she proved the following result: 

\begin{theorem}[Mourgues]\label{thm:mourgues} Let $(K,\cL)$ be a $P$-minimal field. The following are equivalent:
\begin{enumerate}
\item $K$ has definable Skolem functions. 
\item $K$ has cell decomposition over the class of continuous $\cL$-definable functions. 
\end{enumerate}
\end{theorem}

Let us now recall what definable Skolem functions are. A structure $M$ has definable Skolem functions if every definable set admits a definable section. A definable set $X\subseteq M^{n+1}$ has a definable section if there is a definable function $g:\pi(X)\to M$ such that $(x,g(x))\in X$ for all $x\in \pi(X)$, where $\pi$ denotes the projection of $M^{n+1}$ onto the first $n$ coordinates.   

Recently, Darni\`ere and Halupczok \cite{darniereETAL:2015} characterized $P$-minimal structures having cell preparation over the class of continuous definable functions using an additional condition called ``the extreme value property''. This property requires that every continuous definable function from a closed and bounded definable set $X \subseteq K$ to $\Gamma_K$ attains a maximal value. Their theorem can be stated as follows: 

\begin{theorem}[Darni\`ere-Halupczok]\label{thm:DH} Let $(K,\cL)$ be a $P$-minimal field. The following are equivalent:
\begin{enumerate}
\item $K$ has definable Skolem functions and satisfies the extreme value property;  
\item $K$ has function preparation over the class of continuous definable functions. 
\end{enumerate}
\end{theorem}

The existence of $P$-minimal structures without definable Skolem functions was left open in both \cite{mou-09} and \cite{darniereETAL:2015}. In the next section we present a $P$-minimal field $(K,\cL_A)$ where $K$ is a non-standard $p$-adically closed field and the language $\cL_A$ is $\Lring\cup \{A\}$ with $A$ a binary predicate interpreted in $K$ by an $\cL_{an}$-definable subset. The result confirms the intuition that the existence of definable Skolem functions in $P$-minimal structures is not preserved under taking reducts. Despite this negative result, it is worth noting that variants of cell decomposition results have been proven for general $P$-minimal structures by widening the notion of cell (see for example \cite{cub-dar-lee:15, cubi-leen-2015}).

\section{The example}

As stated in the previous section, $(\QQ_p,\cL_{an})$ is $P$-minimal. Let $(K,\cL_{an})$ be a non-standard elementary extension of $(\QQ_p,\cL_{an})$ and let $\rho\in \Gamma_K$ be such that $\rho>n$ for all $n\in\ZZ$. Let $f:\ZZ_p\to\ZZ_p$ be a transcendental convergent power series with coefficients in $\ZZ_p$. Consider the following set of $\cO_K^2$:
\[
A:=\{(x,y)\in \cO_K^2: v(f(x)-y)>\rho\}. 
\]
Abusing notation, let $A$ be a new binary relation symbol and $\cL_A=\Lring\cup\{A\}$. We set $(K,\cL_A)$ as the expansion of $(K,\Lring)$ where $A$ is interpreted as the set $A$ defined above. Notice that every $\cL_A$-definable set is in particular $\cL_{an}$-definable and, since $(K,\cL_{an})$ is $P$-minimal, we trivially get

\begin{proposition} $(K,\cL_A)$ is $P$-minimal. 
\end{proposition}

We will show that $(K,\cL_A)$ has no definable Skolem functions. To prove this we use the following two lemmas whose proofs are postponed to the the end of the section. 

\begin{lemma}\label{lem:aprox} For every $a\in \cO_K$ there are $\tilde{a}\in \cO_K$ and $b\in \QQ_p$ such that $a=\tilde{a}+b$ and $v(\tilde{a})>n$ for every $n\in \ZZ$. 
\end{lemma}

\begin{lemma}\label{lem:zeroset} Let $g:W\subseteq K\to K$ be an $\cL_A$-definable function. Then there is a polynomial $P(X,Y)\in K[X,Y]$ such that for all $x\in W$, $P(x,g(x))=0$. 
\end{lemma}

\begin{theorem}\label{main} The structure $(K,\cL_A)$ does not have definable Skolem functions. 
\end{theorem}

\begin{proof} Suppose for a contradiction that $g:\cO_K\to K$ is a definable Skolem function for $A$, that is, for all $x\in \cO_K$, $v(f(x)-g(x))>\rho$. Notice that since $v(f(x))\geq 0$ and $v(f(x)-g(x))>\rho$ we must have that $g(x)\in \cO_K$. By Lemma \ref{lem:zeroset}, let $P\in K[X,Y]$ be such that $P(x,g(x))=0$ for all $x\in \cO_K$ with 
\[
P(X,Y)=\sum_{(i,j)} a(i,j)X^iY^j
\]
for $a(i,j)\in K$ and $0\leq i\leq j\leq N$ for some $N\in \NN$. Let $(i_0, j_0)$ be such that $v(a(i_0, j_0))=\min (v(a(i, j)))$. Without loss of generality we may assume that $v(a(i_0, j_0))=0$, since we can multiply $P(X,Y)$ by $\frac{1}{a(i_0, j_0)}$. 

\begin{claim}\label{bigger} For all $x\in \cO_K$, $v(P(x,f(x)))>\rho$. 
\end{claim}

Since $P(x,g(x))=0$, we have that $P(x,f(x))=P(x,f(x))-P(x,g(x))$, then

\begin{align*}
v(P(x,f(x))) & = v\left( \sum_{(i, j)}a(i, j)x^i(f(x))^j -\sum_{(i, j)}a(i, j)x^i(g(x))^j\right) \\
			& =  v\left(\sum_{(i, j)\neq (0,0)}a(i,j)x^i((f(x))^j-(g(x))^j)\right) \\
			& \geq  \min_{(i, j)\neq (0,0)}\{ v(a(i,j)x^i((f(x))^j-(g(x))^j))\}\\
			& =  \min_{(i, j)\neq (0,0)} \{v(a(i,j)x^i(f(x)-g(x))((f(x))^{j-1}+\cdots+(g(x))^{j-1}))\}\\
			& >  \min_{(i, j)\neq (0,0)} \{v(a(i,j)x^i((f(x))^{j-1}+\cdots+(g(x))^{j-1}))\}+\rho\\
			& \geq \rho, \\
\end{align*}
which completes the claim. 

\

We will show that there is $x\in \ZZ_p$ such that $v(P(x,f(x)))\in \ZZ$, contradicting the claim since by assumption $\rho>n$ for all $n\in\ZZ$. First split the set of indices $(i,j)$ in $P$ as follows: 
\[
I:=\{(i,j): \exists n\in \ZZ, v(a(i,j))< n\}, \text{ and }
\]
\[
J:=\{(i,j): \forall n\in \ZZ, v(a(i,j))> n\}. 
\]
By Lemma \ref{lem:aprox}, for $(i,j)\in I$, let $\tilde{a}(i,j)\in K$ and $b(i,j)\in \QQ_p$ be such that $a(i,j)=\tilde{a}(i,j)+b(i,j)$ and $v(\tilde{a}(i,j))>n$ for every integer $n$. 

\begin{claim} For all but finitely many $x\in \QQ_p$ 
\[
v(P(x,f(x)))=v\left(\sum_{(i,j)\in I}b(i,j)x^i(f(x))^j\right)\neq\infty.
\]  
\end{claim}

First notice that for $x\in \QQ_p$
\[
v\left(\sum_{(i,j)\in I}a(i,j)x^i(f(x))^j\right)=v\left(\sum_{(i,j)\in I}\tilde{a}(i,j)x^i(f(x))^j + \sum_{(i,j)\in I}b(i,j)x^i(f(x))^j\right).
\]
Since $v(a(i_0,j_0))=0$ we must have $b(i_0,j_0)\neq 0$. By definition of $I$ and $J$ we have that either $\sum_{(i,j)\in I}b(i,j)x^i(f(x))^j=0$ or $\sum_{(i,j)\in I}b(i,j)x^i(f(x))^j\neq0$ and
\[
v\left(\sum_{(i,j)\in I}a(i,j)x^i(f(x))^j\right)=v\left(\sum_{(i,j)\in I}b(i,j)x^i(f(x))^j\right)<v\left(\sum_{(i,j)\in J}a(i,j)x^i(f(x))^j\right).
\]
Since $f$ is transcendental and $b(i_0,j_0)\neq 0$, we have that $\sum_{(i,j)\in I}b(i,j)x^i(f(x))^j=0$ occurs only for finitely many $x\in\QQ_p$, which shows the claim. 

\

Take $x\in \ZZ_p$ such that  $\sum_{(i,j)\in I}b(i,j)x^i(f(x))^j\neq0$. Then, 
\[
v(P(x,f(x)))=v\left(\sum_{(i,j)\in I}b(i,j)x^i(f(x))^j\right)\in \ZZ, 
\]
since both $x$ and all $b(i,j)$ are in $\QQ_p$. 
\end{proof}

We are now left with the proof of Lemmas \ref{lem:aprox} and \ref{lem:zeroset}. For Lemma \ref{lem:aprox} we use pseudo-Cauchy sequences (pseudo-convergent in Kaplansky's \cite{kaplansky1942}, to which we refer the reader for definitions and basic properties). 

\begin{proof}[Proof of Lemma \ref{lem:aprox}] Let $a$ be an element of $K$. If $v(a)>n$ for all $n\in \ZZ$ set $b=0$. Suppose $v(a)\in \ZZ$. Let $(b_i)_{i\in \NN}$ be a pseudo Cauchy sequence of elements in $\QQ_p$ such that $a$ is a pseudo-limit. By completeness of $\QQ_p$ there is $b\in\QQ_p$ which is also a pseudo-limit of that sequence. Set $\tilde{a}:=a-b$. By definition of pseudo Cauchy sequence, for every $n\in\ZZ$ there is $i\in \NN$ such that both $v(a-a_i)>n$ and $v(b-a_i)>n$, which shows $v(a-b)>n$. 
\end{proof}

The proof of Lemma \ref{lem:zeroset} is a bit more involved. It is based on \emph{resplendent quantifier elimination}, a notion coined by Scanlon (see \cite{Scanlon99}) which can be traced back to the work of Pas on relative quantifier elimination for henselian valued fields (see \cite{pas-89}). Let us informally explain what this notion means. For a formal exposition we refer the reader to \cite{rideau15}. 

We will work in a multi-sorted extension $\cL_{\RV^*}^K$ of $\Lring$ (which will be defined later) for which the theory $T=Th(K,\cL_{\RV^*}^K)$ will \emph{relatively eliminate valued field quantifiers}. Denote by $\VF$ the valued field sort. For $T$ to relatively eliminate $\VF$-quantifiers means that every $\cL_{\RV^*}^K$-formula is equivalent modulo $T$ to a $\VF$-quantifier free $\cL_{\RV^*}^K$-formula (which might still have quantifiers for variables in sorts different from $\VF$). We say a language $\cL$ extends $\cL_{\RV^*}^K$ \emph{resplendently over $\VF$} if only new relation and function symbols are added whenever they do not involve the sort $\VF$. Finally, $T$ eliminates $\VF$-quantifiers \emph{resplendently over $\VF$}, if the elimination of $\VF$-quantifiers also holds for any language $\cL$ extending $\cL_{\RV^*}^K$ resplendently over $\VF$. This will be the content of Proposition \ref{RQE}. Let us now introduce the language $\cL_{\RV^*}^K$. We will use the notation choice from \cite{flenner2011}.  

Given $\delta\geq 0$ in $\Gamma_K$, let $\cM_\delta$ denote the ideal $\{x\in K: v(x)>\delta\}$. The $\RV_\delta$ structure is the quotient group
\[
\RV_\delta:=K^\times/(1+\cM_\delta),
\] 
and $\rv_\delta:K^\times\to \RV_\delta$ is the quotient map. We include an element $\infty$ in $\RV_\delta$ and extend $\rv_\delta$ to $K$ setting $rv_\delta(0)=\infty$. Given $\gamma,\delta\in \Gamma_K$ such that $\delta\leq\gamma$, we also denote by $\rv_{\delta}$ the natural map $\rv_\delta:\RV_\gamma\to\RV_\delta$. A partial sum is induced in $RV_\delta$ as the following ternary relation:
\[
\oplus(a_1,a_2,a_3) \Leftrightarrow \exists x_1,x_2,x_3\in K\left(\bigwedge_{i=1}^3 rv_\delta(x_i)=a_i \wedge x_1+x_2=x_3\right).
\]
Let $\cL_0$ be the language $\{\times,\oplus\}$. 
The multi-sorted language $\cL_{\RV^*}^K$ is given by 
\[
\cL_{\RV^*}^K:=
\begin{cases}
(\VF,\Lring)  \\
\text{$(RV_\delta,\cL_{0})$ for each $\delta\geq 0$ in $\Gamma_K$}  \\
\text{$rv_\delta:VF\to RV_\delta$ for each $\delta\geq 0$ in $\Gamma_K$} \\
\text{$rv_\delta:RV_\gamma\to RV_\delta$ for all $0\leq\delta\leq \gamma$ in $\Gamma_K$}. \\
\end{cases}
\]
We use the notation $(K,\cL_{\RV^*}^K)$ for the whole $\cL_{\RV^*}^K$-structure on $K$. Notice that the set of sorts is fixed by $\Gamma_K$. Resplendent relative elimination of $\VF$-quantifiers for the theory of $(K,\cL_{\RV^*}^K)$ is based on relative quantifier elimination results by Basarab \cite{basarab:91}, Kuhlmann \cite{kuhlmann:94} and a more recent accounts by Flenner \cite{flenner2011} and Rideau \cite{rideau15}. 

\begin{proposition}\label{RQE} Let $\cL$ be a language extending $\cL_{\RV^*}^K$ resplendently over $\VF$. Then any $\cL$-formula is equivalent modulo $Th(K,\cL)$ to an $\cL$-formula without $\VF$-quantifiers. 
\end{proposition}

\begin{proof}
This natural extension of Proposition 4.3 in \cite{flenner2011} follows from a careful analysis of the proof given in \cite{flenner2011}. Alternatively, one can follow the classical techniques from Pas and Denef \cite{pas-89, denef-86} (which were inspired by methods of Cohen) and iteratively apply semi-algebraic preparation to the polynomials involved in a given $\cL$-formula in order to eliminate $\VF$-quantifiers. Finally, the result also follows from recent techniques introduced in \cite{rideau15}, which we omit since they will require a much longer exposition.  
\end{proof}

In our case, we use this resplendence to have a better control of $\cL_A$-definable functions in one variable. Consider the image of $A$ by $\rv_\rho$, that is, 
\[
rv_\rho(A)=\{(a,b)\in RV_\rho^2: \exists x,y\in K(rv_\rho(x)=a\wedge rv_\rho(y)=b\wedge (x,y)\in A\}.
\]
For $H$ a new binary symbol let $\cL_H:=\cL_0\cup\{H\}$. Finally let $\cL$ be the extension of $\cL_{\RV^*}^K$ in which for $\rho$ we replace $\cL_0$ by $\cL_H$ in the sort $\RV_\rho$. We set $(K,\cL)$ as the expansion of $(K,\cL_{\RV^*}^K)$ in which $H$ is interpreted as $\rv_\rho(A)$. The two structures are related as follows:

\begin{lemma}\label{lem:bidef}
For every $n$, every $\cL_A$-definable subset of $K^n$ is also $\cL$-definable. 
\end{lemma}

\begin{proof}
Let $X$ be a subset of $K^n$. If $X$ is defined by an $\cL_A$-formula $\phi$, the $\cL$-formula $\psi$ which arises from replacing uniformly the predicate $A$ by $\rv_\rho^{-1}(H)$ also defines $X$. Indeed, $A = \rv_\rho^{-1} (H)$. The left-to-right inclusion is trivial. For the converse, let $(z,w)\in \rv_\rho^{-1}(H)$ and $(x,y)\in A$ such that $\rv_\rho(x)=\rv_\rho(z)$ and $\rv_\rho(y)=\rv_\rho(w)$. This implies that $v(z-x)>\rho$ and $v(y-w)>\rho$, so in particular $(z,w)\in\cO_K^2$. It remains to show that $v(f(z)-w)>\rho$. By definition of $A$ we have that $v(f(x)-y)>\rho$, which by the ultrametric inequality implies that $v(f(x)-w)>\rho$. Again, by the ultrametric inequality, it suffices to show that $v(f(x)-f(z))>\rho$. Since $f$ (in $\ZZ_p$) is defined by a convergent power series with coefficients in $\ZZ_p$, it is easy to see that $v(x-z)\leq v(f(x)-f(z))$ for all $x,z\in \mathbb{Z}_p$, and hence such statement also holds in $K$. Thus, since $v(x-z)>\rho$ we have that $v(f(x)-f(z))>\rho$. Therefore $(z,w)\in A$. 
\end{proof}

The final ingredient in the proof of Lemma \ref{lem:zeroset} is the well-behavior of dimension for definable sets in $P$-minimal fields. The dimension of a definable set $X\subseteq K^n$ is the maximal non-negative integer $k\leq n$ such that there is a projection $\pi:K^n\to K^k$ for which $\pi(X)$ has non-empty interior. For $X=\emptyset$ we set its dimension as $-\infty$. Notice that a definable set is finite if and only if it has dimension less or equal than 0. Given definable sets $X_1,\ldots,X_n$, it was proven in \cite{haskellETAL:10} that 
\[
\dim(X_1\cup\cdots\cup X_n)=\max (\dim(X_i):1\leq i\leq n).
\]
It also follows from the results in \cite{haskellETAL:10} that the dimension is additive: given definable sets $S\subseteq K^n$ and $X\subseteq S\times K^m$ such that the projection of $X$ onto the first $n$-coordinates equals $S$, and such that fibers have fixed dimension $\dim(X_s)=l\leq m$ for all $s\in S$, one has that $\dim(X)=\dim(S)+l$ (see \cite{cub-dar-lee:15}). In particular, the graph of a function $f:X\subseteq K^n\to K$ must have dimension less than $n+1$. For $X_1,\ldots,X_n$ open definable sets in $K^n$, the ultrametric inequality imposes furthermore that $\dim(X_1\cap\cdots \cap X_n)$ is either $n$ or $-\infty$. We have now all ingredients to prove Lemma \ref{lem:zeroset}:

\begin{proof}[Proof of Lemma \ref{lem:zeroset}] Let $G$ be the graph of $g$. By Lemma \ref{lem:bidef}, $G$ is definable by an $\cL$-formula $\phi(x,y)$. By Proposition \ref{RQE}, $\phi$ is equivalent to an $\cL$-formula of the form 
\[
\xi(x,y):=\bigvee_{i\in I} \bigwedge_{j\in J} P_{w}(x,y)=0 \wedge Q_{w}(x,y)\neq 0 \wedge  \theta_{w}(t_{1}^w,\ldots, t_{n_w}^w)
\]
where $w\in I\times J$, $P_{w}, Q_{w}\in K[X,Y]$, $\theta_{w}(x_1,\ldots, x_{n_w})$ is an $\cL$-formula where all variables range over $\RV$ sorts and each term $t_{l}^w$ is of the form 
\[
t_l^w=\rv_{\gamma_{w,l}}(F_{w,l}(x,y)), \text{ with  $\gamma_{w,l}\in \Gamma_K$ and $F_{w,l}\in K[X,Y]$}.
\] 
Notice the result follows if for each $i\in I$ there exists  $j\in J$ such that $P_{(i,j)}\neq 0$. For any polynomial $F\in K[X,Y]$ and any $\delta\geq 0$ in $\Gamma_K$, we have that 
\[
rv_{\delta}(F(x,y))=\infty\Leftrightarrow F(x,y)=0.
\] 
Hence, possibly by replacing the formula $\xi(x,y)$ by an equivalent formula, we may assume that $\theta_w(t_1^w,\ldots, t_{n_w}^w)$ defines a set of dimension 2 or $-\infty$ in $K^2$. Indeed, if $\theta_w(t_1^w,\ldots, t_{n_w}^w)$ defines a set of dimension 1, such a set is already contained in the union of the zero sets of the polynomials $F_{w,l}$.

Let $i\in I$ be such that the formula 
\[
\xi_i(x,y):=\bigwedge_j P_{(i,j)}(x,y)=0 \wedge Q_{(i,j)}(x,y)\neq 0 \wedge \theta_{(i,j)}(t_{1}^{(i,j)},\ldots, t_{n_{(i,j)}}^{(i,j)}),
\] 
defines a non-empty subset of $K^2$. Suppose for a contradiction that for all $j\in J$, the polynomial $P_{(i,j)}$ is the zero polynomial. By our assumption, every formula $\theta_{(i,j)}$ defines a subset of either dimension 2 or dimension $-\infty$ in $K^2$. Formulas of the form $Q_{(i,j)}(x,y)\neq 0$ always define a subset of dimension 2 in $K^2$. Therefore $\xi_i(x,y)$ defines a subset of dimension $2$ or $-\infty$. Since $\xi_i$ defines a non-empty set, its dimension must be 2. This implies $G$ has dimension 2, but since $G$ is the graph of a function, by additivity it must have dimension 1, a contradiction. So for each $i\in I$, there exists $j\in J$ with $P_{(i,j)}$ not the zero polynomial. 
\end{proof}

Let us finish with a remark and two questions. Our proofs used essentially that the $p$-adically closed field $(K,\cL_A)$ of Theorem \ref{main} is non-standard. Moreover, despite such structure does not have Skolem functions, it has a $P$-minimal expansion that does have, namely $(K,\cL_{an})$. This two facts naturally induce the following questions: 

\begin{question} Does any $P$-minimal expansion of $\QQ_p$ (or a finite extension) has definable Skolem functions? 
\end{question}

\begin{question} Does every $P$-minimal field has an expansion with definable Skolem functions?
\end{question}

\subsection*{Acknowledgment} 
Special thanks to Raf Cluckers for many helpful discussions and suggestions. We also thank the referee for his/her useful comments and Eva Leenknegt and Silvain Rideau for many valuable exchanges. The research leading to these results has received funding from the European Research Council,
ERC Grant nr. 615722, MOTMELSUM, 2014--2019. The first author also partially supported by the ERC project TOSSIBERG (grant agreement 637027).

\bibliographystyle{acm}

\bibliography{biblioV2}

\end{document}